\documentclass[]{article}

\usepackage{amsfonts}
\usepackage{amscd}
\usepackage{epsfig}
\usepackage{epic,eepic}
\usepackage{graphicx}
\usepackage[all]{xy}
\usepackage{xypic}
\usepackage{psfrag}
\usepackage{amsmath}
\usepackage{amsthm}
\usepackage{setspace}
\usepackage{url}
\usepackage{here}

\title{Log Minimal Model Program for the Kontsevich Space of Stable Maps $\overline{\mathcal M}_{0,0}(\mathbb P^{3}, 3)$}
\author{Dawei Chen}
\date{}

\newtheorem{theorem}{Theorem}[section]

\newtheorem{corollary}[theorem]{Corollary}

\newtheorem{remark}[theorem]{Remark}

\begin{document}
\bibliographystyle{plain}
\maketitle

\begin{abstract}
We run the log minimal model program for the Kontsevich space of stable maps $\overline{\mathcal M}_{0,0}(\mathbb P^{3}, 3)$ and give modular interpretations to all the intermediate spaces appearing in the process. In particular, we show that one component of the Hilbert scheme $\mathcal H_{3,0,3}$ is the flip of  $\overline{\mathcal M}_{0,0}(\mathbb P^{3}, 3)$ over the Chow variety. Finally as an easy corollary we obtain that $\overline{\mathcal M}_{0,0}(\mathbb P^{3}, 3)$ is a Mori dream space.
\end{abstract}

\tableofcontents

\section{Introduction}
Recently in higher dimensional geometry, with the success of the minimal model program, people became interested in carrying out this program
explicitly for various moduli spaces. In \cite{Ha}, Hassett initiated a framework to understand the canonical model
Proj \big($\bigoplus_{m\geq 0} H^{0}(mK_{\overline{\mathcal M}_{g}})$\big) of the moduli space $\overline{\mathcal M}_{g}$ of stable genus $g$ curves, by using
the log canonical models Proj \big($\bigoplus_{m\geq 0} H^{0}(m(K_{\overline{\mathcal M}_{g}}+\alpha\delta))$\big), where $\delta$ is the total boundary divisor corresponding to singular curves. The eventual goal is to decrease $\alpha$ from 1 to 0 and describe the log canonical models appearing in this process. Surprisingly, it seems that all these models should be geometrically meaningful spaces, by which we mean for instance, moduli spaces parameterizing curves with suitable singularities, or GIT quotients of some Hilbert schemes and Chow varieties of curves. This phenomenon has been already verified for beginning cases, c.f. \cite{Ha}, \cite{HH} and \cite{HL}. But the whole picture is still far from complete. \\

Similarly, we can consider running the log minimal model program on the Kontsevich space of stable maps $\overline{\mathcal M}_{0,0}(\mathbb P^{d}, d)$ instead of $\overline{\mathcal M}_{g}$. $\overline{\mathcal M}_{0,0}(\mathbb P^{d}, d)$ is a useful compactification for the scheme parameterizing smooth degree $d$ rational curves in $\mathbb P^{d}$. One important reason we choose to study $\overline{\mathcal M}_{0,0}(\mathbb P^{d}, d)$ is because its effective cone is completely known, due to the result in \cite[Thm 1.5]{CHS2}. When $d$ is 2, it is well-known that $\overline{\mathcal M}_{0,0}(\mathbb P^{2}, 2)$ is isomorphic to the blow up of the Hilbert scheme of plane conics along the locus of  double lines, i.e., $\overline{\mathcal M}_{0,0}(\mathbb P^{2}, 2) \cong \mbox{Bl}_{V_{2}}\mathbb P^{5}$, where $V_{2}$ is the image of the degree 2 Veronese embedding $\mathbb P^{2} \hookrightarrow \mathbb P^{5}$. So the first nontrivial case is $d = 3$.  \\

In this article, we run the log minimal model program for $\overline{\mathcal M}_{0,0}(\mathbb P^{3}, 3)$ and give geometric interpretations to all the intermediate models. Before we state the main result, let us recall some basic divisors on $\overline{\mathcal M}_{0,0}(\mathbb P^{3}, 3)$. One can refer to \cite{CHS1}, \cite{CHS2} and \cite{P} for a general description of the divisor theory on $\overline{\mathcal M}_{0,0}(\mathbb P^{d}, d)$. \\

$H$ is the class of the divisor of maps whose images intersect a fixed line in $\mathbb P^{3}$; $\Delta = \Delta_{1,2}$ is the class of the boundary divisor consisting of maps with reducible domains; $T$ is the class of the tangency divisor whose complement consists of maps $[C, \mu]$ such that $\mu^{-1}(\Pi)$ is reduced, where $\Pi$ is a fixed plane in $\mathbb P^{3}$; $D_{deg}$ is the class of the degenerate divisor
parameterizing a stable map whose set theoretic image does not span $\mathbb P^{3}$. Finally, fix a point $p$ and a plane $\Lambda$ containing $p$ in $\mathbb P^{3}$. Define a divisor $F$ as the closure of $\{ [C,\mu] \in \mathcal M_{0,0}(\mathbb P^{3}, 3)\ |\ \exists\ p_{1}, p_{2} \in \mu(C)\cap \Lambda\ \mbox{such that $p, p_{1}$ and $p_{2}$ are collinear} \}$. We call $(p\in \Lambda)$ the defining flag for $F$. 

\begin{theorem}
\label{div} 
\ \\
(i). Pic $(\overline{\mathcal M}_{0,0}(\mathbb P^{3}, 3)) \otimes \mathbb Q$ is generated by $H$ and $\Delta$; \\
(ii). The effective cone of $\overline{\mathcal M}_{0,0}(\mathbb P^{3}, 3)$ is generated by $\Delta$ and $D_{deg}$; \\
(iii). The nef cone of $\overline{\mathcal M}_{0,0}(\mathbb P^{3}, 3)$ is generated by $H$ and $T$; \\
(iv). $T = \frac{2}{3}(H + \Delta)$; \\
(v). $D_{deg} = \frac{2}{3}(H-\frac{1}{2}\Delta)$; \\
(vi). $F = \frac{5}{3}(H - \frac{1}{5}\Delta).$
\end{theorem}

\begin{proof}
All of the above have been proved in \cite{CHS1}, \cite{CHS2} and \cite{P} in a more general setting except (vi).
Write $F$ as the linear combination $aH + b\Delta$. Take two test curves in $\overline{\mathcal M}_{0,0}(\mathbb P^{3}, 3)$: \\
\  $B_{1}$: a pencil of lines attached to a plane conic at the base point. \\
\  $B_{2}$: a pencil of plane conics attached to a line at one of the four base points. \\

The following intersection numbers are easy to verify: 
$$H\ldotp B_{1} = 1,\ \Delta \ldotp B_{1} = -1,\ F\ldotp B_{1} = 2, $$
$$H\ldotp B_{2} = 1, \ \Delta \ldotp B_{2} = 2, \ F\ldotp B_{2} = 1. $$

Hence, we get $a = \frac{5}{3}$ and $b = -\frac{1}{3}$.
\end{proof}

The picture below shows the decomposition of the effective cone of $\overline{\mathcal M}_{0,0}(\mathbb P^{3}, 3)$ by these divisors, which was first told to the author by Izzet Coskun. 

\begin{figure}[H]
    \centering
    \psfrag{T}{$T$}
    \psfrag{H}{$H$}
    \psfrag{E}{$\Delta$}
    \psfrag{F}{$F$}
    \psfrag{D}{$D_{deg}$}
    \includegraphics[scale=0.6]{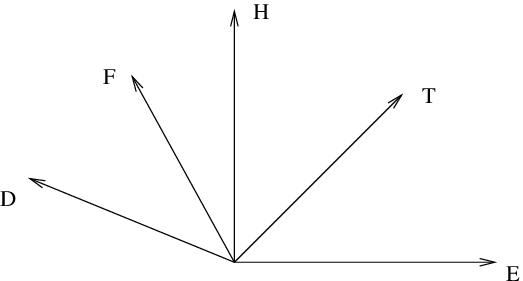}
\end{figure}

To run the log minimal model program for $\overline{\mathcal M}_{0,0}(\mathbb P^{3}, 3)$, since the effective cone is 2-dimensional, we only need to adjust the parameter $\alpha$ in Proj \big($\bigoplus_{m\geq 0} H^{0}(m(H+\alpha\Delta))$\big) from $-\frac{1}{2}$ to $\infty$ to cover the whole effective cone. Denote $\overline{M}(\alpha)$ as the model Proj \big($\bigoplus_{m\geq 0} H^{0}(m(H+\alpha\Delta))$\big). In order to give a geometric interpretation for $\overline{M}(\alpha)$, we also need a few other compactifications besides $\overline{\mathcal M}_{0,0}(\mathbb P^{3}, 3)$ for the space of twisted cubics. \\

Consider the Hilbert scheme $\mathcal H_{3,0,3}$ of degree 3 and arithmetic genus 0 curves in $\mathbb P^{3}$. It consists of two smooth irreducible components $\mathcal H$ and $\mathcal H'$, of dimension 12 and 15 respectively, c.f. the main theorem in \cite{PS}. Only the component $\mathcal H$ contains the locus of twisted cubics. $\mathcal H'$ contains the points corresponding to plane cubic curves union a point in $\mathbb P^{3}$. $\mathcal H$ admits a natural map $g$ to the Chow variety by forgetting the scheme structure of a curve but only remembering its cycle class. Denote the image of $g$ by $\mathcal{C}$. Let $N\subset \mathcal H$ be the locus of curves possessing at least one nonreduced
primary component. $N$ is irreducible and a general element in $N$ corresponds to a double line of genus $-1$ with a line meeting it and lying in its projective tangent space at the
point of intersection, c.f. \cite[p. 39]{H} and \cite[4.4]{Lee}. It is clear that Exc($g$) = $g^{-1}(N)$ and dim $N=9$. \\

Similarly, there is also a natural morphism $f$ from $\overline{\mathcal M}_{0,0}(\mathbb P^{3}, 3)$ to $\mathcal{C}$ by forgetting the maps and only remembering the image cycles. Define the locus of the multi-image maps $M\subset \overline{\mathcal M}_{0,0}(\mathbb P^{3}, 3)$ as \{[$C, \mu] \in \overline{\mathcal M}_{0,0}(\mathbb P^{3}, 3)\ |\ \exists$ a component $C_{1}$ of $C$ such that the map degree of $\mu$ restricted to $C_{1}$ is greater than 1\}. Here the map degree stands for the degree of the covering, which is different from the degree of the image cycle in $\mathbb P^{3}$. 
$M$ breaks into two components $M_{1,2}$ and $M_{3}$. A general point in $M_{1,2}$ corresponds to a map whose domain consists of a nodal union of two $\mathbb P^{1}$'s. One component maps with degree 1 and the other one maps with degree 2. A point in $M_{3}$ is a map whose domain is a single $\mathbb P^{1}$ which maps with degree 3. Apparently $M$ is contained in the degenerate divisor $D_{deg}$. We also have Exc($f$) = $f^{-1}(M)$, dim $M_{1,2} = 9$ and dim $M_{3} =8$. \\

Note that there is a birational map $\phi: \overline{\mathcal M}_{0,0}(\mathbb P^{3}, 3) \dashrightarrow \mathcal H$ which is a canonical isomorphism away from Exc($f$) and Exc($g$). In particular, $\phi$ is an isomorphism in codimension two, so Pic($\overline{\mathcal M}_{0,0}(\mathbb P^{3}, 3))\otimes\mathbb Q$ is isomorphic to Pic($\mathcal H)\otimes\mathbb Q$. We will keep the same notation for divisors on both spaces. However, one should be aware that $D_{deg}$ on $\mathcal H$ only consists of curves whose primary components as schemes are planar. For instance, a triple line whose ideal is defined by the square of the ideal of a line is not contained in $D_{deg}$.  \\

A twisted cubic $C$ can be parameterized by a unique net of quadrics $H^{0}(\mathcal I_{C}(2))$. So the space of twisted cubics can be realized as an open subscheme in the Grassmannian $\mathbb G(2, 9)$. Denote its closure by $\mathcal H(2)$. Moreover, one can check that $H^{0}(\mathcal I_{C}(2))$ is always 3-dimensional for any point
$[C] \in \mathcal H$, by using \cite[Lemma 2]{PS}. Hence, we get a morphism $h: \mathcal H\rightarrow \mathcal H(2)$. The space $\mathcal H(2)$ and the morphism $h$ have been investigated intensively
in \cite{EPS}, so we are able to insert them into our program. \\

The last space we need here is less well-known. It is the space of 2-stable maps $\overline{\mathcal M}_{0,0}(\mathbb P^{3}, 3, 2)$, which is originally raised up in \cite{MM} and also studied later in \cite{Pa}. Roughly speaking, the difference between $\overline{\mathcal M}_{0,0}(\mathbb P^{3}, 3)$ and $\overline{\mathcal M}_{0,0}(\mathbb P^{3}, 3, 2)$ is to replace a tail of degree 1 by a base point
at the attaching point of the tail. There is a contraction map $\theta: \overline{\mathcal M}_{0,0}(\mathbb P^{3}, 3)\rightarrow \overline{\mathcal M}_{0,0}(\mathbb P^{3}, 3, 2)$ contracting the boundary
$\Delta$. \\

After these preparations, now we state the main result.
\begin{theorem}
\label{main} 
\ \\
(i). $\overline{M}(\alpha) \cong \overline{\mathcal M}_{0,0}(\mathbb P^{3}, 3)$ for $\alpha \in (0,1)$; \\
(ii). $\overline{M}(0) \cong \mathcal{C}$ and $f: \overline{\mathcal M}_{0,0}(\mathbb P^{3}, 3) \rightarrow \mathcal{C}$ is a small contraction contracting the locus of the multi-image maps $M$; \\
(iii). $g: \mathcal H \rightarrow \mathcal{C}$ is the flip of $f$ associated to the divisor $H+\alpha\Delta$ and $\overline{M}(\alpha)\cong\mathcal H$ for $\alpha \in (-\frac{1}{5}, 0)$;\\
(iv). $h: \mathcal H \rightarrow \mathcal H(2)$ is a divisorial contraction contracting the locus of the degenerate curves $D_{deg}$ and $\overline{M}(\alpha)\cong \mathcal H(2)$ for
$\alpha \in (-\frac{1}{2}, -\frac{1}{5}]$; \\
(v). $\theta: \overline{\mathcal M}_{0,0}(\mathbb P^{3}, 3)\rightarrow \overline{\mathcal M}_{0,0}(\mathbb P^{3}, 3, 2)$ is a divisorial contraction contracting the boundary $\Delta$
and $ \overline{M}(\alpha)\cong\overline{\mathcal M}_{0,0}(\mathbb P^{3}, 3, 2)$ for $\alpha \in [1, \infty)$.
\end{theorem}

The above result can be best seen from the following diagram: 

\begin{figure}[H]
    \centering
    \psfrag{M2}{$\overline{\mathcal M}_{0,0}(\mathbb P^{3}, 3, 2)$}
    \psfrag{M3}{$\overline{\mathcal M}_{0,0}(\mathbb P^{3}, 3)$}
    \psfrag{C}{$\mathcal C$}
    \psfrag{H}{$\mathcal H$}
    \psfrag{H2}{$\mathcal H(2)$}
    \psfrag{theta}{$\theta$}
    \psfrag{f}{$f$}
    \psfrag{g}{$g$}
    \psfrag{h}{$h$}
    \psfrag{phi}{$\phi$}
    \psfrag{flip}{$\mbox{flip}$}
    \includegraphics[scale=0.5]{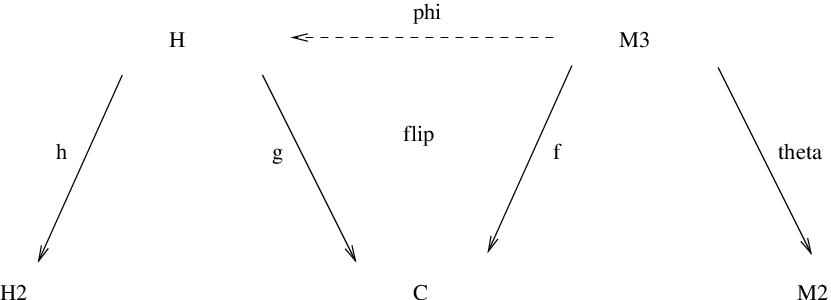}
\end{figure}

This paper is organized as follows. In Section 2 we study the base locus of an effective divisor on $\overline{\mathcal M}_{0,0}(\mathbb P^{3}, 3)$. The main theorem is proved in
Section 3. Some corollaries and further questions are discussed in the last section. Throughout, we work over an algebraic closed field of characteristic 0. All the divisors are $\mathbb Q$-divisors and 
the equivalence relation means numerical equivalence. \\

{\bf Acknowledgements.} This work was originally motivated by conversations with Izzet Coskun and my advisor Joe Harris. I am grateful to both of them. 
Without their help, this paper could not be finished so smoothly. Especially I would like to thank Izzet Coskun who told me about the chamber decomposition of the effective cone and 
generously let me continue working on the rest of the project. I also want to thank David Smyth for numerous helpful comments and suggestions.

\section{Base Loci of Effective Divisors on $\overline{\mathcal M}_{0,0}(\mathbb P^{3}, 3)$}
In order to study the log canonical model associated to an effective divisor $K$ which is proportional to $H+\alpha\Delta$ on $\overline{\mathcal M}_{0,0}(\mathbb P^{3}, 3)$, it would be useful to
figure out the base locus B($K$) first. Here the base locus for a $\mathbb Q$-divisor means the stable base locus, in the sense of \cite[Def 2.1.20]{L}. 
In the introduction section, we already saw a chamber decomposition for the effective cone of $\overline{\mathcal M}_{0,0}(\mathbb P^{3}, 3)$.
Actually B($K$) only depends on which chamber $K$ lies in.

\begin{theorem}
\label{BL} 
\ \\
(i). $K$ is base point free if it lies in the chamber bounded by $H$ and $T$; \\
(ii). B($K$) is the boundary $\Delta$ if $K$ lies in the chamber bounded by $T$ and $\Delta$; \\
(iii). B($K$) is the multi-image locus $M$ if $K$ lies in the chamber bounded by $H$ and $F$; \\
(iv). B($K$) is the locus of the degenerate maps $D_{deg}$ if $K$ lies in the chamber bounded by $F$ and $D_{deg}$.
\end{theorem}

The picture below provides a quick view of the above result. 

\begin{figure}[H]
    \centering
    \psfrag{T}{$T$}
    \psfrag{H}{$H$}
    \psfrag{F}{$F$}
    \psfrag{D}{$D_{deg}$}
    \psfrag{E}{$\Delta$}
    \psfrag{ample}{$B=\emptyset$}
    \psfrag{BL=E}{$B=\Delta$}
    \psfrag{BL=M}{$B=M$}
    \psfrag{BL=D}{$B=D_{deg}$}
    \includegraphics[scale=0.9]{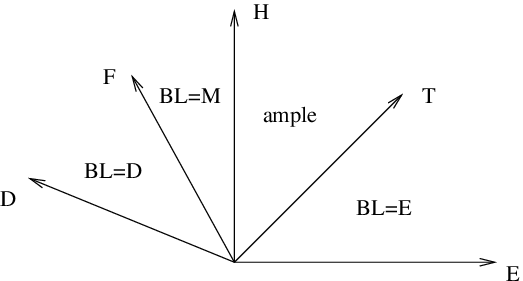}
\end{figure}

\begin{proof}
(i) is trivial since the nef cone of $\overline{\mathcal M}_{0,0}(\mathbb P^{3}, 3)$ is generated by $H$ and $T$. $H$ and $T$ are base point free by their definitions. \\

For (ii), take the test curve $B_{1}$ we used in the proof of Theorem \ref{div}. $T\ldotp B_{1} = \frac{2}{3}(H + \Delta)\ldotp B_{1} = 0$ and $\Delta \ldotp B_{1} =  -1$ so $K\ldotp B_{1} < 0$. Such $B_{1}$'s cover an open set of $\Delta$. Hence, $\Delta \subset$ B($K$). After removing the base locus $\Delta$ from $K$, only the part $T$ is left and $T$ is base point free. \\

For (iii), if $[C,\mu]$ is not a multi-image map, we can always exhibit a line defining $H$ and a flag defining $F$ such that
$[C,\mu]$ is neither contained in $H$ nor in $F$. Hence, B($K)\subset M$. On the other hand, take two general quadrics $Q_{1}, Q_{2}$ on a plane $\Lambda$
and a point $p\in \Lambda$. For a line $L_{\lambda}$ in the pencil of lines passing through $p$ on $\Lambda$, define a degree 2 map from
$L_{\lambda}$ to a line $L = [X, Y, 0, 0]\subset \mathbb P^{3}$ by $L_{\lambda}\rightarrow [Q_{1}|_{L_{\lambda}}, Q_{2}|_{L_{\lambda}}, 0, 0]$. If $L_{\lambda}$ is spanned by $p$ and a base point $b \in Q_{1}\cap Q_{2}$, we need to blow up $b$ and the domain curve breaks to a nodal union of two lines each of which maps isomorphically to $L$. We attach another line to this pencil at $p$ and map it isomorphically to a line $L'$ in $\mathbb P^{3}$ which
passes through $[Q_{1}(p), Q_{2}(p), 0, 0]$. Then we obtain a 1-dimensional family $E_{1,2}$ in $M_{1,2}$. $E_{1,2}\ldotp H = 0$ and $E_{1,2}\ldotp \Delta = 3$, so $E_{1,2}\ldotp (aF+bH) < 0$ for any
$a>0$. Moreover, such $E_{1,2}$'s sweep out an open set of $M_{1,2}$. Therefore $M_{1,2}$ must be contained in B($K$). Furthermore, take two general cubics $S_{1}, S_{2}$ on $\Lambda$. For a line $L_{\lambda}$ in the pencil of lines passing through $p$ on $\Lambda$, define a degree 3 map from
$L_{\lambda}$ to $L = [X, Y, 0, 0]\subset \mathbb P^{3}$ by $L_{\lambda}\rightarrow [S_{1}|_{L_{\lambda}}, S_{2}|_{L_{\lambda}}, 0, 0]$. If $L_{\lambda}$ is spanned by $p$ and a base point $b \in S_{1}\cap S_{2}$, we blow up $b$ and the domain curve breaks to a nodal union of two lines. One of them maps 2:1 to $L$ and the other maps 1:1 to $L$. Now we get a 1-dimensional family $E_{3}\subset M_{3}$.
$E_{3}\ldotp H = 0$ and $E_{3}\ldotp \Delta = 8$, so $E_{3}\ldotp (aF+bH) < 0$ for any $a>0$. Moreover, such $E_{3}$'s cover an open set of $M_{3}$. Therefore $M_{3}$ must be contained in B($K$). \\

For (iv), take a pencil $R$ of plane rational nodal cubics all passing through a fixed node $p$. Obviously we have $R\ldotp H = 1$. Besides $p$, there are five other base points $b_{1},\ldots, b_{5}$ of this pencil.
If $C$ is a degree $3$ reducible plane curve singular at $p$ which also passes through $b_{1},\ldots, b_{5}$, then $C$ must be the union of a conic and a line. The conic goes through $p$ and four other base points while the line goes through $p$ and the remaining base point. Hence, there are five reducible curves in this pencil, which implies that  $R\ldotp \Delta = 5$.
Another way to calculate the intersection is by blowing up $p$ and $b_{1}, \ldots, b_{5}$ from the plane to get a surface $S$ as a rational fibration over $\mathbb P^{1}$. $\chi(S) = \chi(\mathbb P^{2}) + 6 = 9$, but $\chi(S)$ also equals $\chi(\mathbb P^{1})\cdot \chi(\mathbb P^{1}) +$ the number of reducible fibers. Hence, we get 5 reducible curves in $R$ again.
Therefore, $R \ldotp (aD_{deg}+bF) = \frac{1}{3}(2a+5b)(R\ldotp H) - \frac{1}{3}(a+b)(R\ldotp \Delta) = -a < 0$ for any $a > 0$, which implies $D_{deg}\subset$ B($K$). Now remove $D_{deg}$ from $K$ and only
the component $F$ is left. We already verified in (iii) that B($F$) = $M$ and $M\subset D_{deg}$.
\end{proof}

\section{Log Canonical Models for $\overline{\mathcal M}_{0,0}(\mathbb P^{3}, 3)$}
This section is devoted entirely to prove Theorem \ref{main} step by step. Recall that an effective divisor $K$ on $\overline{\mathcal M}_{0,0}(\mathbb P^{3}, 3)$
can be written as $H + \alpha \Delta$ up to scalar. We adjust the parameter $\alpha$ to consider the model Proj \big($\bigoplus_{m\geq 0} H^{0}(m(H+\alpha\Delta))$\big) associated to $K$.

\begin{proof}
(i) is trivial since the ample cone of $\overline{\mathcal M}_{0,0}(\mathbb P^{3}, 3)$ is the interior of the nef cone generated by $H$ and $T$. \\

For (ii), the divisor $H$ is nef and base point free. Obviously it contracts the multi-image locus $M$, so the resulting space is $\mathcal C$.
Moreover, the two components $M_{1,2}$ and $M_{3}$ have dimension 9 and 8 respectively. Hence, it is a small contraction. \\

(iii) is the central part of the whole proof. The main point here is to check that $K = H + \alpha \Delta$ is $g$-ample and $-K$ is $f$-ample for $\alpha \in (-\frac{1}{5},0)$.  \\

Firstly, we verify that $-K$ is $f$-ample. Pick a cycle class
$[C]$ in the image of $\mbox{Exc}(f)$ in $\mathcal C$. $[C]$ can be either
$[L_{1}]+2[L_{2}]$ or $3[L_{3}]$, where $L_{1}, L_{2}, L_{3}$ are
lines in $\mathbb P^{3}$ and $L_{1}, L_{2}$ intersect at one node
$p$. \\

If $[C] = [L_{1}] + 2[L_{2}]$, then the fiber $f^{-1}([C])$
is contained in $M_{1,2}$ and is isomorphic to the locus of the
evaluation divisor $\mathcal L = ev^{*}(p)$ inside
$\overline{\mathcal M}_{0,1}(L_{2}, 2)$. The marked point
maps to $p$ for the purpose of attaching a copy of $L_{1}$. For a
general point corresponding to a map $(\mathbb P^{1}, \mu)$ in
$\mathcal L$, $\mu$ branches at two distinct points $p_{1}, p_{2}$
away from $p$ on $L_{2}$. $\mu^{-1}(p)$ consists of two points as
the candidates for the marked point. However, the involution of $\mu$
switches these two points. So in this case the branch points
$p_{1}, p_{2}$ determine uniquely an element in $\mathcal L$.
When $p_{1}$ and $p_{2}$ meet away from $p$, the domain of $\mu$
breaks to a nodal union of two $\mathbb P^{1}$'s. Both of them
map isomorphically to $L_{2}$. The marked point can be contained in either
one of the two domain components, but the involution of $\mu$ switches them.
The most special case is when $p_{1}, p_{2}$ and $p$ all coincide.
Then the domain of $\mu$ splits to a length-3 chain of $\mathbb
P^{1}$'s. The middle $\mathbb P^{1}$ contains the marked point and
gets contracted to $p$. The other two rational tails
both map isomorphically to $L_{2}$. Therefore, $\mathcal L$ is
isomorphic to $\mbox{Sym}^{2}\mathbb P^{1}\cong \mathbb P^{2}$. Now we
take the family $E_{1,2}$ used in the proof of Theorem \ref{BL}
(iii). $E_{1,2}$ is an effective curve class in $f^{-1}([C])$ since the base
point of the pencil maps to a fixed point in $\mathbb P^{3}$. There are 4
reducible curves in this pencil, so $E_{1,2}\ldotp K = \alpha
E_{1,2}\ldotp \Delta = 4\alpha < 0$. \\

Similarly, if $[C] =
3[L_{3}]$, the fiber $f^{-1}([C])$ is isomorphic to
$\overline{\mathcal M}_{0,0}(\mathbb P^{1}, 3)$. Pic
($\overline{\mathcal M}_{0,0}(\mathbb P^{1}, 3))\otimes\mathbb Q$ is generated by
the boundary divisor $\Delta_{1,2}$. So $\overline{\mbox{NE}}(f^{-1}([C]))$ is
1-dimensional. We pick the effective curve class $E_{3}$ in Theorem \ref{BL}
(iii).
$E_{3}\ldotp H = 0$ and $E_{3}\ldotp \Delta = 8$, so $E_{3}\ldotp K = 8\alpha < 0$. Overall, we proved that $-K$ is $f$-ample. \\

Next, we prove that $K$ is $g$-ample. If $[C] = [L_{1}] + 2[L_{2}]$ is in the image of $\mbox{Exc}(g)$ in $\mathcal C$, $g^{-1}([C])$ consists of four different types of curves up to projective equivalence in $\mathcal H$, c.f. \cite[p. 39-41]{H} and \cite[4.4]{Lee}: \\
\  (I). A double line $D$ of genus $-1$ supported on $L_{2}$ meeting $L_{1}$ at one point $p$ such that $L_{1} \subset \mathbb T_{p}D$. \\
\  (II). A planar double line $D$ of genus 0 supported on $L_{2}$ meeting $L_{1}$ at one point $p$ and $L_{1}$ not lying on the plane containing $D$. \\
\  (III). A planar double line $D$ of genus 0 supported on $L_{2}$ meeting $L_{1}$ at one point $p$ with an embedded point
$q\neq p$ somewhere on $L_{2}$ and $L_{1}$ lying on the plane containing $D$. \\
\  (IV). A planar double line $D$ of genus 0 supported on $L_{2}$ meeting $L_{1}$ at one point $p$ with an embedded point
at $p$ and $L_{1}$ lying on the plane containing $D$. \\

Note that (I) can specialize to either (II) or (III), and all of them can specialize to (IV).  Now we do an explicit calculation to show that $g^{-1}([C])\cong \mathbb P^{2}$. \\

Suppose the coordinate of $\mathbb P^{3}$ is given by $[X, Y, Z, W]$.
$L_{1}$ is defined by $Y=Z=0$, $L_{2}$ is defined by $X=Y=0$, and $p = L_{1}\cap L_{2} = [0,0,0,1]$. The double line $D$ in (I) is defined by the ideal
$\mathcal I_{D} = ((X,Y)^{2}, XG-YF)$, where $F = aZ + bW$ and $G=cZ + dW$ are two linear forms on $L_{2}$, c.f. \cite[Prop 1.4]{N}. Since the tangent direction $[F,G,0,0]$ at $p$
is contained in the plane $Y=0$ spanned by $L_{1}, L_{2}$, we get $G(p) = 0$, i.e., $d=0$ and $G = c Z$. When $b, c$ are both nonzero, it corresponds to
(I). If $c=0$ but $b\neq 0$, $\mathcal I_{D\cup L_{1}} = ((X,Y)^{2}, Y(aZ+bW))\cap (Y,Z) = (X, Y, aZ + bW)^{2}\cap (Y, X^{2})\cap (Y, Z)$, so an embedded point appears at the point $[0,0,-b,a]$ on $L_{2}$, and $L_{1}$ is on the plane where the planar double line $D$ lies. Hence, we get case (III). If $c\neq 0$ but $b=0$,
$\mathcal I_{D\cup L_{1}} = ((X,Y)^{2}, Z(aY-cX))\cap (Y,Z) = (Y,Z)\cap (aY-cX, Y^{2})$, that is, case (II). Finally, if $b=c=0$, then
$\mathcal I_{D\cup L_{1}} = ((X,Y)^{2}, YZ)\cap (Y,Z) = (Y,Z)\cap (Y,X^{2})\cap (X,Y,Z)^{2}$, which is case (IV).
It is clear that $g^{-1}([C])$ is isomorphic to $\mathbb P^{2}$ defined by the parameters $[a,b,c]$ up to scalar.
The locus (I) is isomorphic to $\mathbb P^{2}\backslash \mbox{(II)}\cup\mbox{(III)}\cup\mbox{(IV)}$. (II) and (III) are both isomorphic to $\mathbb A^{1}$ and (IV) is their intersection point. We can pick a general line class $R$ in $g^{-1}([C])$. $R \ldotp H = 0$ and $R \ldotp D_{deg} = 1$ since (I) and (II) are not contained in $D_{deg}$ but (III) is. Hence, $R\ldotp K > 0$ for $K$ as a positive linear combination of $H$ and $D_{deg}$. \\

 For the other case, if $[C] = 3[L_{3}]$, $g^{-1}([C])$ in $\mathcal H$ is the locus of triple lines supported on $L_{3}$. Again, there are three different types: \\
\  (V). A triple line $J$ supported on $L_{3}$ lying on a quadric cone. \\
\  (VI). A planar triple line $J$ supported on $L_{3}$ with an embedded point on it. \\
\  (VII). A triple line $J$ whose ideal is given by the square of $\mathcal I_{L_{3}}$. \\

We can show that $g^{-1}([C])$ is isomorphic to a projective bundle $\mathbb P(\mathcal O \oplus \mathcal O \oplus \mathcal O(-3))$ contracting a section over $\mathbb P^{1}$, and $D_{deg}$
restricted to $g^{-1}([C])$ is ample. \\

For $J$ in case (V), suppose that $L_{3}$ is defined by $X=Y=0$ and the multiplicity two structure $\mathcal I_{2}$ associated to $J$ is $(\beta X - \alpha Y, (X,Y)^{2})$, where $\beta X - \alpha Y$ is the tangent plane of a quadric cone containing $J$ at $L_{3}$. If $\beta \neq 0$, let $t = \alpha/\beta$. Then $\mathcal I_{2}$ can be rewritten as $(X-tY, Y^{2})$.  As in \cite[Prop 2.1]{N}, $I_{J}$ has the form $((X-tY)^{2}, (X-tY)Y, Y^{3}, (aW+bZ)(X-tY) - c Y^{2})$ which is parameterized by
$(t; [a,b,c])$. Similarly, if $\alpha \neq 0$, let $s = \beta/\alpha$. Then $I_{J} = ((sX-Y)^{2}, (sX-Y)X, X^{3}, (aW+bZ)(sX-Y) - cX^{2})$ which is parameterized by $(s; [a,b,c])$. When we change the coordinate, it is easy to check that $(t; [a,b,c])\rightarrow (1/t ; [a,b, c / t^{3}])$. So after gluing together,
we simply get a $\mathbb P^{2}$ bundle $\mathbb PV^{*} \cong \mathbb P(\mathcal O\oplus \mathcal O\oplus \mathcal O(-3))$ over $\mathbb P^{1}$. The locus $S$ where $c=0$ is isomorphic to $\mathbb P^{1}\times \mathbb P^{1}$ in $\mathbb PV^{*}$ which corresponds to case (VI), and $S$ is also the restriction of $D_{deg}$. However, if $a=b=0$ but $c\neq 0$, then $\mathcal I_{J}$ is always equal to $(X,Y)^{2}$ independent of $[\alpha, \beta]$. It implies that $g^{-1}([C])$ is isomorphic to $\mathbb PV^{*}$ contracting the section $\Gamma$ cut out by $a=b=0$.
$\Gamma$ is one extremal ray of $\overline{\mbox{NE}}(\mathbb PV^{*})$. $S\ldotp \Gamma = 0$ since $S$ is physically away from $\Gamma$. So $S$ is nef on $\mathbb PV^{*}$. After contracting $\Gamma$, $S$ becomes ample on $g^{-1}([C])$. Again, $R\ldotp K > 0$ for any effective curve class $R\in \overline{\mbox{NE}}(g^{-1}([C]))$, since $K$ is a positive linear combination of $H$ and $D_{deg}$. \\

For (iv), since we already flipped $\overline{\mathcal M}_{0,0}(\mathbb P^{3}, 3)$ to $\mathcal H$ for $\alpha \in (-\frac{1}{5}, 0)$, to resume the log canonical model program we need to
start from $\mathcal H$. If $\alpha = -\frac{1}{5}$, $K$ is proportional to $F$. In the proof of Theorem \ref{BL} (iv), we know that $F\ldotp R$ = 0, where $R$ is a pencil of plane rational nodal curves on $\Lambda$ possessing a common node $p$. The intersection number $F\ldotp R$ = 0 also holds on $\mathcal H$, although a spatial embedded point occurs at $p$ for each curve in $R$.
So after contracting the class $R$, the resulting space should only remember the information of the flag $(p\in \Lambda)$. Actually this is the space of nets of quadrics $\mathcal H(2)$ we defined in the introduction section. Recall that there is a natural morphism $h$ from the open set of $\mathcal H$ corresponding to twisted cubics to the Grassmannian $\mathbb G(2,9)$
by $h([C]) = [H^{0}(\mathcal I_{C}(2))]\in \mathbb PH^{0}(\mathcal O_{\mathbb P^{3}}(2)).$ $\mathcal H(2)$ is just the closure of the image of $h$ in $\mathbb G(2,9)$. This map can extend to $\mathcal H$, since for any $[C]\in \mathcal H$, $h^{0}(\mathcal I_{C}(2)) = 3$. Moreover, $[C], [C']\in D_{deg}$ share the same flag $(p\in \Lambda)$ if and only if $H^{0}(\mathcal I_{C}(2)) = H^{0}(\mathcal I_{C'}(2))$, c.f. \cite[Lemma 2]{PS}. Hence,
$h^{*}(\mathcal O_{\mathcal H(2)}(1))\ldotp R = 0$. $h^{*}(\mathcal O_{\mathcal H(2)}(1))$ must be proportional to $F$, since $F\ldotp R=0$ and the Picard number of $\mathcal H$ is 2. It also implies that $F$ is nef on $\mathcal H$. So the nef cone of $\mathcal H$ is generated by $H$ and $F$. If $\alpha < -\frac{1}{5}$, $K$ is proportional to $F + aD_{deg}$ for some $a>0$. Since $D_{deg}$ is in the base locus of $K$, after we remove it, $K$ becomes proportional to $F$. \\

For (v), when $\alpha = 1$, $K$ is proportional to the tangency divisor $T$. We already know that $T$ is nef on $\overline{\mathcal M}_{0,0}(\mathbb P^{3}, 3)$. If an extremal ray $R$ satisfies the relation $R\ldotp T = 0$, $R$ must be contained in the boundary $\Delta$. Moreover, if an image point of a map in $R$ is the image of a node, then it cannot move. Otherwise it must meet the plane $\Pi$ defining $T$ and $R\ldotp T$ would not be zero. Similarly, the image of a domain component which maps with degree greater than 1 cannot move. Otherwise 
the locus of branch points would meet $\Pi$ at some point. Hence, only the images of components of map degree 1 can vary. On the other hand, we can choose a general plane defining $T$ so that this plane is away from the images of the finitely many attaching points of those components which map with degree 1, and also away from the branch points on the fixed images of components which map with degree greater than 1. Then clearly we have $T\ldotp R = 0$. Therefore, this divisorial contraction associated to $T$ contracts all degree 1 tails but only remembers their attaching points. The resulting space has already appeared in \cite{MM} and \cite{Pa}, the space of 2-stable maps $\overline{\mathcal M}_{0,0}(\mathbb P^{3}, 3, 2)$, c.f. Remark \ref{MM}. If $\alpha >1$, $K$ is proportional to $T + b\Delta$ for some $b>0$. Since $\Delta$ is in the base locus of $K$, after removing it, $K$ becomes proportional to $T$.
\end{proof}

\begin{remark}
Since $H\ldotp R = 0$ for any effective curve class $R\in \overline{\mbox{NE}}(g^{-1}([C]))$, where $[C]\in \mathcal C$ is contained in the image of Exc($g$), and $H, F$ generate the nef cone of $\mathcal H$, we know that $F$ restricted to $g^{-1}([C])$
must be ample. Now the ampleness of $D_{deg}$ on $g^{-1}([C])$ follows immediately since $F$ is a positive linear combination of $H$ and $D_{deg}$. This provides a quick proof
for the $g$-ampleness of $K$ in (iii). Similarly, one can verify that $T$ is ample when restricted to $f^{-1}([C])$, which also implies that $-K$ is $f$-ample.
\end{remark}

\begin{remark}
Take a plane $\Lambda$ and three general points $p_{1}, p_{2}, p_{3}$ on it. Define a divisor $G$ on $\mathcal H$ as the closure of twisted cubics whose intersection points
with $\Lambda$ are contained in a conic that also contains $p_{1}, p_{2}$ and $p_{3}$. The attention to this divisor was pointed out to the author by Joe Harris. Let us determine the class of $G$. Using the test families $B_{1}$ and $B_{2}$ in
the proof of Theorem \ref{div}, we get the following intersection numers:
$$H\ldotp B_{1} = 1,\ \Delta \ldotp B_{1} = -1,\ G\ldotp B_{1} = 2, $$
$$H\ldotp B_{2} = 1, \ \Delta \ldotp B_{2} = 2, \ G\ldotp B_{2} = 1, $$
which tells that $G = \frac{5}{3}(H-\frac{1}{5}\Delta) = F$. We know that $F$ is proportional to $h^{*}(\mathcal O_{\mathcal H(2)}(1))$ in the proof of Theorem \ref{main} $(iv)$. Actually it is easy to check
$G =  h^{*}(\mathcal O_{\mathcal H(2)}(1))$ directly. $\mathcal O(1)$ of the Grassmannian $\mathbb G(2,9)$ is defined by the Schubert cycle $\sigma_{1}$. When it pulls back to
$\mathcal H$, the locus corresponds to twisted cubics whose nets of quadrics always meet a fixed $\mathbb P^{6}$ of quadrics. We can take this $\mathbb P^{6}$ naturally by imposing $p_{1}, p_{2}$
and $p_{3}$ on quadrics. So after restricted to $\Lambda$, those nets of quadrics become nets of conics. At least one conic in each net comes from the fixed $\mathbb P^{6}$ of quadrics
containing $p_{1}, p_{2}, p_{3}$. This is just the divisor $G$ by its definition.
\end{remark}

\begin{remark}
As mentioned in the introduction section, the geometry of $\mathcal H(2)$ and the morphism $h: \mathcal H\rightarrow\mathcal H(2)$ have already been studied in \cite{EPS}. Actually $\mathcal H(2)$ is smooth
and $h$ is the blow up of $H(2)$ along the locus of the point-plane incidence correspondence $\{(p\in \Lambda)\}$. Moreover, the Betti numbers of $\mathcal H$ and $\mathcal H(2)$ are calculated in \cite{EPS}. In particular, the Picard number of $\mathcal H$ and $\mathcal H(2)$ is 2 and 1 respectively. This coincides with our result, since $h$ is a divisorial contraction which makes the Picard number drop by 1. 
\end{remark}

\begin{remark}
\label{MM}
The same argument as in the proof of Theorem \ref{main} (v) also applies to $\overline{\mathcal M}_{0,0}(\mathbb P^{r}, d).$ $T$ is nef and the associated divisorial contraction contracts all degree 1 tails. So the resulting space is the space of $(d-1)$-stable maps $\overline{\mathcal M}_{0,0}(\mathbb P^{r}, d, d-1)$ in \cite{MM} and \cite{Pa}. Reader should watch out that by a tail we mean a component of a rational curve whose removal does not disconnect 
the curve. For instance, the middle $\mathbb P^{1}$ of a length-3 chain of $\mathbb P^{1}$'s is a backbone rather than a tail. For reader's convenience, we include the definition for the space of $k$-stable maps $\overline{\mathcal M}_{0,0}(\mathbb P^{r}, d, k)$ as follows. One can refer to \cite[Def 1.2]{MM} and \cite[Cor 4.6]{Pa} for more details. \\

Let $k$ be a natural number, $0 \leq k \leq d$. Fix a rational number $\epsilon$ such that $0 < \epsilon < 1$. The moduli space of $k$-stable maps $\overline{\mathcal M}_{0,0}(\mathbb P^{r}, d, k)$ parameterizes 
the following data $(\pi: C\rightarrow S,\ \mu: C\rightarrow \mathbb P^{r},\ \mathfrak L,\ e)$ where: \\
(1) $\pi: C\rightarrow S$ is a flat family of rational nodal curves over the scheme $S$. \\
(2) $\mathfrak L$ is a line bundle on $C$ of degree $d$ on each fiber $C_{s}$ which, together with the morphism $e: \mathcal O_{C}^{r+1} \rightarrow \mathfrak L$ determines the rational map 
$\mu: C\rightarrow \mathbb P^{r}$. \\
(3) $\omega_{C/S}^{d-k+\epsilon}\otimes \mathfrak L$ is relatively ample over $S$.  \\
(4) $\mathfrak G :=$ coker $e$, restricted to each fiber $C_{s}$, is a skyscraper sheaf, and dim $\mathfrak G_{p} \leq d-k$ for any $p\in C_{s}$, where $\mathfrak G_{p}$ is the stalk of $\mathfrak G$ at $p$. 
If $0 < \mbox{dim}\ \mathfrak G_{p}$, then $p\in C_{s}$ is a smooth point of $C_{s}$. \\

Condition (3) rules out those tails which map with degree less or equal to $d-k$. Instead, we allow the appearance of base points by (4). There is a natural morphism    
$\overline{\mathcal M}_{0,0}(\mathbb P^{r}, d, k+1) \rightarrow \overline{\mathcal M}_{0,0}(\mathbb P^{r}, d, k)$ which contracts the tails of degree $d-k$, c.f. \cite[Prop 1.3]{MM}. For the contraction associated to the divisor $T$, we get $\overline{\mathcal M}_{0,0}(\mathbb P^{r}, d)\rightarrow \overline{\mathcal M}_{0,0}(\mathbb P^{r}, d, d-1)$, which replaces a tail of degree 1 by a simple base point at its attaching point. So only the boundary $\Delta_{1,d-1}$ is contracted. \\

A more general contraction contracting the total boundary of $\overline{\mathcal M}_{0,0}(\mathbb P^{r}, d)$ is considered in \cite{CHS1} without mentioning the modular interpretation for the resulting space, c.f. \cite[Thm 1.9]{CHS1}.   
\end{remark}

\section{$\overline{\mathcal M}_{0,0}(\mathbb P^{3}, 3)$ is a Mori dream space}
The chamber decomposition for the effective cone of $\overline{\mathcal M}_{0,0}(\mathbb P^{3}, 3)$ also holds for the effective cone of $\mathcal H$.
The nef cone of $\mathcal H$ is generated by $H$ and $F$. For an effective divisor $K$ on $\mathcal H$, we have a similar result about the base locus of $K$, compared with Theorem \ref{BL}.

\begin{theorem}
\label{BLH}
\ \\
(i). $K$ is base point free if it lies in the chamber bounded by $H$ and $F$; \\
(ii). B($K$) is the boundary $\Delta$ if $K$ lies in the chamber bounded by $T$ and $\Delta$; \\
(iii). B($K$) is the locus of the nonreduced curves $N$ if $K$ lies in the chamber bounded by $H$ and $T$; \\
(iv). B($K$) is the locus of the degenerate curves $D_{deg}$ if $K$ lies in the chamber bounded by $F$ and $D_{deg}$.
\end{theorem}

This theorem can be verified easily by using the same method in the proof of Theorem \ref{BL}. We leave it to the reader. \\

Now it is clear that the flip from $\overline{\mathcal M}_{0,0}(\mathbb P^{3}, 3)$ to $\mathcal H$ replaces the locus of the multi-image maps $M$ by the locus of the nonreduced curves $N$. It also switches the ample cone bounded by $H$ and $T$ to the one bounded by $H$ and $F$. An interesting fact is that the flipping locus $M = M_{1,2}\cup M_{3}$ is reducible but $N$ is irreducible. \\

As a corollary, we have the following conclusion.
\begin{corollary}
$\overline{\mathcal M}_{0,0}(\mathbb P^{3}, 3)$ is a Mori dream space. Its Cox ring is finitely generated and Mori's program can be carried out for any divisor on $\overline{\mathcal M}_{0,0}(\mathbb P^{3}, 3)$, i.e. the necessary contractions and flips exist and any sequence terminates.  
\end{corollary}

\begin{remark}
The same results hold for the Hilbert scheme $\mathcal H$. 
\end{remark}

The nef cone of $\overline{\mathcal M}_{0,0}(\mathbb P^{3}, 3)$ with the pull back of the nef cone from $\mathcal H$ together
consist of all the effective divisors without stable base components. Hence, $\overline{\mathcal M}_{0,0}(\mathbb P^{3}, 3)$ is a Mori dream space, c.f. \cite[1.10]{HK}.
The fact that its Cox ring is finitely generated is just a property of Mori dream spaces. Moreover, Theorem \ref{main} tells us how to run the Mori's program explicitly for an effective divisor on $\overline{\mathcal M}_{0,0}(\mathbb P^{3}, 3)$. \\

{\bf Question:} It would be interesting to understand the ring structure of the Cox ring for $\overline{\mathcal M}_{0,0}(\mathbb P^{3}, 3)$.

Department of Mathematics, Harvard University, 1 Oxford Street, Cambridge, MA 02138 \par
{\it Email address:} dchen@math.harvard.edu

\end{document}